\documentclass[preprint,12pt]{elsarticle}

\usepackage[T1]{fontenc}
\usepackage{comment}
\usepackage{color}
\usepackage{hyperref}
\usepackage{ulem}

\usepackage{amssymb}
\usepackage{amsmath}
 \usepackage{amsthm}

\theoremstyle{plain}
\newtheorem{lemma}{Lemma}
\newtheorem{thm}{Theorem}

\newtheorem{proposition}{Proposition}

\theoremstyle{definition}

\newtheorem{example}{Example}
\newtheorem{remark}{Remark}

\journal{}

\begin{document}

\begin{frontmatter}



\title{Universal Bound on the Eigenvalues of 2-Positive Trace-Preserving Maps}


\author{Frederik vom Ende${}^a$} 

\affiliation{organization={Dahlem Center for Complex Quantum
Systems, Freie Universit\"at Berlin},
            addressline={Arnimallee 14}, 
            city={Berlin},
            postcode={14195}, 
            country={Germany}}

\author{Dariusz Chru{\'s}ci{\'n}ski${}^b$} 

\affiliation{organization={Institute of Physics, Faculty of Physics, Astronomy and Informatics, and Nicolaus Copernicus University},
            addressline={Grudzi\k{a}dzka 5/7}, 
            city={Toru\'n},
            postcode={87-100}, 
            country={Poland}}

\author{Gen Kimura${}^c$} 
\affiliation{organization={College of Systems Engineering and Science, Shibaura Institute of Technology},
            city={Saitama},
            postcode={330-8570}, 
            country={Japan}}

\author{Paolo Muratore-Ginanneschi${}^d$} 

\affiliation{organization={Department of Mathematics and Statistics, University of Helsinki},
            addressline={P.O. Box 68}, 
            city={Helsinki},
            postcode={FIN-00014}, 
            country={Finland}}

\begin{abstract}
We prove an upper bound on the trace of any 2-positive, trace-preserving map in terms of its smallest eigenvalue. We show that this spectral bound is tight, and that 2-positivity is necessary for this inequality to hold in general. Moreover, we use this to infer a similar bound for generators of one-parameter semigroups of 2-positive trace-preserving maps. With this approach we generalize known results for completely positive trace-preserving dynamics while providing a significantly simpler proof that is entirely algebraic.
\end{abstract}



\begin{keyword}
2-positivity \sep Complete positivity \sep Quantum-dynamical semigroup \sep eigenvalue inequality \sep relaxation rates
%
\MSC[2020] 15A42 \sep 15B48 \sep 47D06 \sep 81P47 \sep 81Q10
\end{keyword}

\end{frontmatter}



\section{Introduction}

Complete positivity plays a central role both in the dynamics of quantum systems \cite{Kraus73,BreuPetr02,Davies76,RH12} and in quantum information theory in general \cite{NC10,Watrous18,Holevo12}.
It requires that a map sends positive semi-definite operators to positive semi-definite operators, even when acting only on a part of a composite system (in a tensor-product sense, cf.~Sec.~\ref{sec:main} below for the precise definition), thus ensuring that the map represents a valid physical process \cite{Paulsen03}.
Despite being such a central notion to quantum physics, little has long been known about the constraints this imposes on the eigenvalues of such maps---apart from some basic observations, e.g., that the eigenvalues come in complex conjugate pairs, or that the spectral radius of positive trace-preserving maps is always $1$.
In fact, the inverse eigenvalue problem for quantum channels (i.e., completely positive maps which are additionally trace-preserving) is notoriously difficult, and it is fully solved only for the 2-dimensional (``qubit'') case \cite{WPG10}.

From a physics point of view, investigating this connection is motivated by the fact that the eigenvalues of generators of quantum dynamics are in one-to-one relation to so-called \textit{relaxation rates} of the underlying physical process (see below for more details).
These relaxation rates can be measured in the laboratory, and they encode information about fundamental physical processes, such as thermalization, equilibration, decoherence, and dissipation \cite{RH12,AlickiLendi07,SZ20}.
Early bounds on these relaxation rates date to the 1970s \cite{GKS76}, and for qubit systems such a bound was established in \cite{KimG2002}. An optimal inequality for quantum systems of arbitrary dimension was only recently conjectured \cite{CKKS21,CFKO21} and subsequently proved \cite{MKC25} via Lyapunov-theoretic methods.
The precise statement from \cite{MKC25}, translated back to eigenvalues, is that any $d$-dimensional Markovian generator \mbox{$L:\mathbb C^{d\times d}\to\mathbb C^{d\times d}$}---i.e., $L$ is linear, and $e^{tL}$ is completely positive and trace-preserving (CPTP) for all $t\geq 0$---satisfies
\begin{equation}\label{eq:intro0}
    {\rm tr}(L)\leq d\min\Re(\sigma(L))\,,
\end{equation}
where $\min\Re(\sigma(L))$ is the smallest real part of any eigenvalue of $L$.
In the language of relaxation rates (i.e., $\Gamma_j:=-\Re(\lambda_j)$, where $\lambda_j$ denotes the eigenvalues of $L$) this result takes the original form \cite{MKC25}: \mbox{$\Gamma_j\leq\frac1d\sum_k\Gamma_k$}. 
Note that this inequality is tight, as it can be achieved by a generator $L$ of a Markovian completely positive and trace-preserving map \cite{CKKS21}.
We also note that recently, similar inequalities for qubit Schwarz dynamics \cite{CKM24}
have been obtained.

Building on these developments, the aim of this work is to show that complete positivity is not necessary for \eqref{eq:intro0} to hold, despite the fact that equality is achieved by a completely positive map.
In fact, in Theorem~\ref{thm_main2} we show that this bound continues to hold even when relaxing complete positivity to 2-positivity (again, see Sec.~\ref{sec:main} for precise definitions).
This is most interesting for quantum processes with time-dependent generators (i.e., \mbox{$\dot\Phi(t)=L(t)\circ\Phi(t)$} instead of $\dot\Phi(t)=L\circ\Phi(t)$) as there, the intermediate maps 
{$V(t,s)$}
which satisfy $\Phi(t)= {V(t,s)}\circ\Phi(s)$ may not be completely positive anymore, but often satisfy weaker notions of positivity (which have different physical implications) \cite{Wolf08a},
\cite[Sec.~7]{Chrus22}.

Our approach is fundamentally different from the original proof of the completely positive case \cite{MKC25}: using now purely algebraic methods, in Theorem~\ref{thm_main} we first establish 
an inequality on the trace of arbitrary 2-positive trace-preserving maps $\Phi$:
\begin{equation*}
    {\rm tr}(\Phi)\leq d\min\Re(\sigma(\Phi))+d^2-d\,,
\end{equation*}
and only in a simple second step we translate this to the generator case.
Thus, not only is our proof more basic and direct, we also find a constraint on the eigenvalues of 2-positive maps (as opposed to generators of such maps)---something which has not been considered previously.
Finally, we will discuss some possible directions and open questions in Section~\ref{sec_outlook}.
 
\section{Main Results}\label{sec:main}

Before presenting our main results, we first establish some notation.
Throughout this work, $\mathcal L(\mathbb C^{d\times d})$ stands for the 
set of all linear maps on $\mathbb C^{d\times d}$, $\sigma(\cdot)$ denotes 
the spectrum of a linear operator, and the symbol $\Re$ stands for real 
part of a complex number. A matrix $X\in\mathbb C^{d\times d}$ will be 
called \textit{positive semi-definite}---denoted by $X\geq 0$---if 
$\langle \psi,X\psi\rangle\geq 0$ for all $\psi\in\mathbb C^d$ (we shall 
use the convention that the inner product is linear in the second 
argument, as is common in mathematical physics); because we work over the 
complex numbers, $X\geq 0$ automatically implies that $X$ is 
Hermitian, i.e., $X^\dagger=X$ where $(\cdot)^\dagger$ denotes the 
adjoint. As is common in the physics literature, we shall also use the notation $\langle \psi | X | \psi \rangle$ in place of $\langle \psi, X \psi \rangle$, and $|x\rangle\langle y|$ for any $x,y\in\mathbb C^d$ will be short-hand for the linear map $z\mapsto \langle y,z\rangle x$ on $\mathbb C^d$.

Some special classes of linear maps:
$\Phi\in\mathcal L(\mathbb C^{d\times d})$ is called \textit{positive}
if $\Phi(A)\geq 0$ for all $A\in\mathbb C^{d\times d}$ positive semi-definite. 
Note that every positive map is \textit{Hermitian-preserving}; that is, if $X$ is Hermitian, then $\Phi(X)$ is Hermitian. (This can be readily seen by expressing any Hermitian matrix $X$ as $X=X_+ - X_-$, where $X_{\pm}$ are positive semi-definite.)
Next, given any $k\in\mathbb N$ one says $\Phi$ is \textit{$k$-positive} if the extended map ${\rm id}_k\otimes\Phi:\mathbb C^{k\times k}\otimes\mathbb C^{d\times d}\to\mathbb C^{k\times k}\otimes\mathbb C^{d\times d}$---uniquely defined via $({\rm id}_k\otimes\Phi)(A\otimes B)=A\otimes\Phi(B)$---is positive.
If $\Phi$ is $k$-positive for all $k\in\mathbb N$, then $\Phi$ is called \textit{completely positive}. This is equivalently characterized by the fact that the \textit{Choi matrix} of $\Phi$
$$
\mathsf C(\Phi):=\sum_{j,k=1}^d|j\rangle\langle k|\otimes\Phi(|j\rangle\langle k|)\in\mathbb C^{d\times d}\otimes\mathbb C^{d\times d}
$$
is positive semi-definite \cite{Choi75,Paulsen03}.
We say that $\Phi$ is \textit{trace preserving} if ${\rm tr}(\Phi(X))={\rm tr}(X)$ for all $X\in\mathbb C^{d\times d}$.
With these definitions, a \textit{quantum channel} is a completely positive, trace-preserving map \cite{NC10,Watrous18,Holevo12}.

The trace for $\Phi\in\mathcal L(\mathbb C^{d\times d})$ is defined in the usual way, i.e., ${\rm tr}(\Phi):=\sum_{j=1}^{d^2}\langle G_j,\Phi(G_j)\rangle_{\rm HS}$ for some orthonormal basis $\{G_j\}_{j=1}^{d^2}$ of $\mathbb C^{d\times d}$ with respect to the Hilbert-Schmidt inner product $\langle A,B\rangle_{\rm HS}:={\rm tr}(A^\dagger B)$.
In fact, starting from any orthonormal basis $\{g_j\}_{j=1}^d$ of $\mathbb C^d$, $\{|g_k\rangle\langle g_j|\}_{j,k=1}^d$ is an orthonormal basis of $(\mathbb C^{d\times d},\langle\cdot,\cdot\rangle_{\rm HS})$ so one has 
\begin{equation}\label{trRep}
{\rm tr}(\Phi)=\sum_{j,k=1}^d\langle g_k|\Phi(|g_k\rangle\langle g_j|)|g_j\rangle.
\end{equation}
Of course, the trace of $\Phi\in\mathcal L(\mathbb C^{d\times d})$ is also the sum of all eigenvalues of $\Phi$, thus affirming that our main results are really spectral constraints on certain subclasses of linear maps.
With all this in mind, let us first state our main results in a precise manner:
\begin{thm}\label{thm_main}
    Let $\Phi\in\mathcal L(\mathbb C^{d\times d})$ be a 2-positive and trace-preserving linear map. Then
    \begin{equation}\label{eq:thm_main_1}
        {\rm tr}(\Phi)\leq d\min\Re(\sigma(\Phi))+(d^2-d).
    \end{equation}
\end{thm}
\if
Notably, this is strictly stronger than the trivial constraint
\begin{equation}\label{eq:trivialbound}
    {\rm tr}(\Phi)=\sum_{\lambda\in\sigma(\Phi)}\Re(\lambda)\leq \min\Re(\sigma(\Phi))+(d^2-1)\max\Re(\sigma(\Phi))
\end{equation}
which is satisfied by all linear maps for which ${\rm tr}(\Phi)$ is real (in particular, this includes all Hermitian-preserving maps).
Moreover, the reason that $\max\Re(\sigma(\Phi))$ does not appear explicitly in~\eqref{eq:thm_main_1} is that the spectral radius of all positive trace-preserving maps is $1$ \cite[Prop.~4.26]{Watrous18} so $\max\Re(\sigma(\Phi))=1$ for the maps we consider.
\fi 
Before presenting the proof, a few remarks are in order:
\begin{remark}\label{rem_boundtight}
    \begin{itemize}
        \item[(i)] Note that there is a trivial bound that holds for all positive, trace-preserving maps $\Phi$:
\begin{equation}\label{eq:trivialbound}
    \rm{tr}(\Phi) \leq \min \Re(\sigma(\Phi)) + (d^2 - 1).
\end{equation}
Indeed, for any map $ \Phi $ whose trace is real (as is the case, for example, for Hermitian-preserving maps), one has $
\rm{tr}(\Phi) = \sum_{\lambda \in \sigma(\Phi)} \Re(\lambda)$,
which is bounded above by $ \min \Re(\sigma(\Phi)) + (d^2 - 1)\max \Re(\sigma(\Phi)) $, corresponding to the worst-case configuration.  
The bound~\eqref{eq:trivialbound} then comes from the well-known fact that the spectral radius of any positive, trace-preserving map is equal to 1~\cite[Prop.~4.26]{Watrous18}, so $\max\Re(\sigma(\Phi)) \leq 1$.
Notably, the bound~\eqref{eq:thm_main_1} is strictly stronger than the trivial bound~\eqref{eq:trivialbound}.
        \item[(ii)] Beware that 2-positivity in Thm.~\ref{thm_main} cannot be weakened further as there exist positive, trace-preserving maps for which~(\ref{eq:thm_main_1}) is violated: 
An elementary counterexample is provided by the transpose map \mbox{$\tau(X) := X^T$} in dimension $d=2$. Indeed, it is well known that the transpose map is
positive but not 2-positive \cite{Choi75}. It has eigenvalues $ 1,-1 $ with multiplicities $ \frac{d(d+1)}{2} $ and $ \frac{d(d-1)}{2} $, respectively. 
Thus, the trace of $\tau$ is $1 \cdot \frac{d(d+1)}{2} + (-1) \cdot \frac{d(d-1)}{2} = d$ so \eqref{eq:thm_main_1} is satisfied if and only if $d \geq 3$, and hence is violated for $d=2$.
\if
The transpose map \sout{{$\tau(X) := X^T$}} in 2 dimensions has trace 2 and eigenvalues ${-1,1}$ so
$$
{\rm tr}(\tau)=2\not\leq 0=2\cdot(-1)+4-2=d\min\Re(\sigma(\tau))+(d^2-d)\,.
$$
\fi
In fact this example can be slightly generalized, refer to Example~\ref{ex_ttc} (\ref{app_a}).
Interestingly, however, in three or more dimensions we could not find any example of a positive trace-preserving map which violates (\ref{eq:thm_main_1}); actually, extensive numerical search suggests that the decomposable maps\footnote{
Recall that $\Phi\in\mathcal L(\mathbb C^{d\times d})$ is called \textit{decomposable} if there exist $\Phi_1,\Phi_2\in\mathcal L(\mathbb C^{d\times d})$ completely positive such that $\Phi=\Phi_1+\Phi_2\circ \tau$ with $\tau$ the usual transpose.
}
satisfy~(\ref{eq:thm_main_1}) as soon as $d\geq 3$---and currently we have no way of generating random (non-decomposable) positive maps
so we have no way to conduct numerics, at least not beyond the standard examples of positive, but not completely positive maps (e.g., Choi map \cite{Choi75b} and its generalizations \cite{Cho1992}, \cite[p.~301]{Bengtsson17}, reduction map \cite{HH99}, Robertson map \cite{Robertson83,Robertson85,CS12}, Osaka maps \cite{Osaka1991}, Breuer-Hall map \cite{Breuer06b,Breuer06,Hall06}, etc. See~\cite[Sec.~4]{CK07}, \cite{GF13,Kye2013} for overviews of such maps).
\item[(iii)]
Another thing to note is that Thm.~\ref{thm_main} is optimal
in the following sense:
For every $c>d$ there exists a 2-positive, trace-preserving map $\Phi$ such that
\begin{equation}\label{eq:optimal_bound}
    {\rm tr}(\Phi)> c \min\Re(\sigma(\Phi))+(d^2-c)\,.
\end{equation}
A corresponding example (Example~\ref{ex_bound_tight}) can be found in \ref{app_a}.
In other words, the ``ratio'' between the smallest and the largest eigenvalue in this upper bound cannot be improved further (without violating 2-positivity).
    \end{itemize}
\end{remark}

Theorem~\ref{thm_main}, while interesting in its own right, also leads directly to the following result for generators.  
As mentioned in the introduction, it generalizes a previously known statement---originally established under the assumption of complete positivity---to the more general framework of 2-positivity, with the additional advantage that the resulting proof is significantly simpler.
\begin{thm}\label{thm_main2}
    Let a generator $L\in\mathcal L(\mathbb C^{d\times d})$ of 2-positive, trace-preserving dynamics be given, i.e., $e^{tL}$ is 2-positive and trace-preserving for all $t\geq 0$. Then
    \begin{equation}\label{eq:thm_main2_1}
        {\rm tr}(L)\leq d\min\Re(\sigma(L))\,.
    \end{equation}
    Equivalently, all relaxation rates $\Gamma_j:=-\Re(\lambda_j)$ of any such $L$ satisfy
    \begin{equation}\label{eq:relax_rates}
        \Gamma_j\leq\frac1d\sum_k\Gamma_k
        \,.
    \end{equation}
\end{thm}

Again, this is strictly stronger than the trivial bound~\eqref{eq:trivialbound} which for the case of positive trace-preserving generators takes the form ${\rm tr}(L)\leq\min\Re(\sigma(L))$. 
This is because positive, trace-preserving maps are trace-norm contractive~\cite{PG06}, so the real parts of all eigenvalues of the corresponding generators must be non-positive.  
Note also that, every such generator $L$ necessarily has $0$ as an eigenvalue~\cite[Prop.~5]{BNT08b}.
\if(because positive trace-preserving maps are trace-norm contractive \cite{PG06} so the real part of all eigenvalues of corresponding generators necessarily are necessarily non-positive; also every such $L$ has $0$ as an eigenvalue \cite[Prop.~5]{BNT08b}).
\fi

\subsection{Proof of Theorem~\ref{thm_main}}
We now turn to the proofs of our main results.
For Theorem~\ref{thm_main}, the key idea is to invoke a classical object: the so-called transition matrices.
More precisely, given $\Phi\in\mathcal L(\mathbb C^{d\times d})$ and any orthonormal basis $G:=\{g_j\}_{j=1}^d$ of $\mathbb C^d$ define $T_G(\Phi)\in\mathbb C^{d\times d}$ (or $T_G$, for short) via
$(T_G)_{jk}:=\langle g_j|\Phi(|g_k\rangle\langle g_k|)|g_j\rangle$.
The reason this is useful is that this allows for a natural bound between the trace of any 2-positive maps and any of its transition matrices:
\begin{lemma}\label{lemma_trPhiTG}
For all 2-positive maps $\Phi\in\mathcal L(\mathbb C^{d\times d})$ and all orthonormal bases $G$ of $\mathbb C^d$ it holds that
\begin{equation}\label{eq:trPhiTG}
    {\rm tr}(\Phi)\leq d\,{\rm tr}(T_G)\,.
\end{equation}
\end{lemma}
\begin{proof}
First we use the expression of the trace \eqref{trRep}:
\begin{align*}
    {\rm tr}(\Phi)&=\sum_{j,k=1}^d\langle g_k|\Phi(|g_k\rangle\langle g_j|)g_j\rangle\\
&=\sum_{k=1}^d\langle g_k|\Phi(|g_k\rangle\langle g_k|)|g_k\rangle +\sum_{j,k=1,j\neq k}^d\langle g_k|\Phi(|g_k\rangle\langle g_j|)g_j\rangle\\
&={\rm tr}(T_G)+\sum_{k=1}^d\sum_{j=k+1}^n2\Re \Bigl( \langle g_k|\Phi(|g_k\rangle\langle g_j|)g_j\rangle\Bigr).
\end{align*}
Now define the vectors $\psi_{jk}^\pm:=|0\rangle\otimes|g_k\rangle\pm|1\rangle\otimes|g_j\rangle\in\mathbb C^2\otimes\mathbb C^d$.
Because $\Phi$ is 2-positive (recall: ${\rm id}_2\otimes\Phi:\mathcal L(\mathbb C^2\otimes\mathbb C^d)\to \mathcal L(\mathbb C^2\otimes\mathbb C^d)$ is positive)
$\langle\psi_{jk}^-|({\rm id}_2\otimes\Phi)( |\psi^+_{jk}\rangle\langle\psi^+_{jk}| )|\psi^-_{jk}\rangle\geq 0$.
But this expression evaluates to
\begin{align*}
   \Big\langle \begin{pmatrix}
       g_k\\-g_j
   \end{pmatrix} \Big|({\rm id}_2\otimes \Phi)& \begin{pmatrix}
       |g_k\rangle\langle g_k|&|g_k\rangle\langle g_j|\\
    |g_j\rangle\langle g_k|&|g_j\rangle\langle g_j|
   \end{pmatrix} \Big| \begin{pmatrix}
       g_k\\-g_j
   \end{pmatrix} \Big\rangle \\
    &=\langle g_k|\Phi(|g_k\rangle\langle g_k|)|g_k\rangle-\langle g_k|\Phi(|g_k\rangle\langle g_j|)|g_j\rangle\\
    &-\langle g_j|\Phi(|g_j\rangle\langle g_k|)|g_k\rangle+\langle g_j|\Phi(|g_j\rangle\langle g_j|)|g_j\rangle\,.
\end{align*}
Thus---because this quantity is non-negative and $\Phi$ is Hermitian-preserving---we find
\begin{align*}
\hspace*{-50pt}2\Re \langle g_k|\Phi(|g_k\rangle\langle g_j|)g_j\rangle&=\langle g_k|\Phi(|g_k\rangle\langle g_j|)|g_j\rangle+\langle g_j|\Phi(|g_j\rangle\langle g_k|)|g_k\rangle\\
&\leq \langle g_k|\Phi(|g_k\rangle\langle g_k|)|g_k\rangle+\langle g_j|\Phi(|g_j\rangle\langle g_j|)|g_j\rangle\\
&=(T_G)_{jj}+(T_G)_{kk}\,.
\end{align*}
Altogether, this yields the desired inequality via
\begin{align*}
 {\rm tr}(\Phi)&={\rm tr}(T_G)+\sum_{j=1}^d\sum_{k=j+1}^n2\Re \langle g_k|\Phi(|g_k\rangle\langle g_j|)g_j\rangle\\
 &\leq {\rm tr}(T_G)+\sum_{j=1}^d\sum_{k=j+1}^d((T_G)_{jj}+(T_G)_{kk})\\
 &={\rm tr}(T_G)+(d-1)\sum_{j=1}^d(T_G)_{jj}=d\,{\rm tr}(T_G)\,.\qquad\qedhere
\end{align*}
\end{proof}

This lemma suggests the following strategy: Given $\Phi$ we could try to find an orthonormal basis $G$ such that ${\rm tr}(T_G)\leq\min\Re(\sigma(\Phi))+(d-1)$, because---combined with~\eqref{eq:trPhiTG}---this would imply (\ref{eq:thm_main_1}) at once.
While it is not clear how to carry out this construction for \textit{all} $\Phi$---assuming this is even possible---we will see that the set of maps for which it can be done is ``large enough'' to establish~\eqref{eq:thm_main_1} for all 2-positive trace-preserving maps.
However, before carrying out this strategy, once again some remarks are in order:
\begin{remark}
\begin{itemize}
    \item[(i)] Unsurprisingly, 2-positivity is necessary for Lemma~\ref{lemma_trPhiTG} to hold: Consider the qubit map $\Phi(X):=\sigma_xX^T\sigma_x$ (with $\sigma_x=|0\rangle \langle 1| + |1\rangle \langle 0|$ being the usual Pauli-$x$ matrix, and $\{|0\rangle,|1\rangle\}$ the standard basis of $\mathbb C^2$). This map is positive but not 2-positive, and ${\rm tr}(\Phi)=2>0=2\cdot 0=2\cdot {\rm tr}(\sigma_x)=2\, {\rm tr}(T)$. In particular, no matter how the pre-factor $d$ on the right-hand side of~(\ref{eq:trPhiTG}) is changed the inequality can never hold.
    \item[(ii)] More interestingly, the inequality in Lemma~\ref{lemma_trPhiTG} does \textit{not} characterize 2-positivity, in the following sense: There exist positive maps which are not 2-positive but which satisfy~(\ref{eq:trPhiTG}) for \textit{all} orthonormal bases $G$, refer to Example~\ref{ex_2p} (\ref{app_a}).
\end{itemize}
\end{remark}

Now for the promised proof.

\begin{proof}[Proof of Thm.~\ref{thm_main}]
We will proceed in three steps.
    \begin{itemize}
        \item[Step 1]: Following the strategy outlined above, we first show that~\eqref{eq:thm_main_1} holds for all 2-positive and trace-preserving maps $\Phi \in \mathcal{L}(\mathbb{C}^{d \times d})$ for which the eigenvalue $\lambda$ of $\Phi$ with the smallest real part is itself real.
    \end{itemize}
In this case there exists $X\in\mathbb C^{d\times d}$ Hermitian such that\footnote{
    This is true for all real eigenvalues $\lambda$ of a Hermitian-preserving linear map: $\Phi(Y)=\lambda Y$ implies $\Phi(Y^\dagger )=\Phi(Y)^\dagger =\overline{\lambda}Y^\dagger =\lambda Y^\dagger $ (because $\lambda$ is assumed to be real) so---unless $Y$ was Hermitian to begin with---$X:=\frac i2(Y-Y^\dagger )\neq 0$ is a Hermitian eigenvector of $\Phi$ to $\lambda$.\label{footnote_Hermitian_evector}
    }
    $\Phi(X)=\lambda X$. Let $G:=\{g_j\}_{j=1}^d$ be any orthonormal basis of $\mathbb C^d$ such that $X=\sum_j x_j|g_j\rangle\langle g_j|$ with $ x:=(x_1,\ldots,x_n)^T\in\mathbb R^d$ the eigenvalues of $X$. Then $T_G x=\lambda x$ because for all $j$
    \begin{align*}
        (T_G x)_j&=\sum_k(T_G)_{jk}x_k=\sum_k \langle g_j|\Phi(x_k|g_k\rangle\langle g_k|)|g_j\rangle\\
        &=\langle g_j|\Phi(X)|g_j\rangle=\lambda  \langle g_j|X|g_j\rangle=\lambda x_j\,.
    \end{align*}
Now $T_G$ having $\lambda=\min\Re(\sigma(\Phi))$ as an eigenvalue allows for the following upper bound: by Lemma~\ref{lemma_trPhiTG}
$$
{\rm tr}(\Phi)\leq d\,{\rm tr}(T_G)=d\sum_{\lambda'\in \sigma(T_G)}\lambda'\leq d\big(\lambda +d-1\big)=d  \min\Re(\sigma(\Phi))+d^2-d\,.
$$
In the second inequality we used that $\max\Re(\sigma(T_G))=1$ because $T_G$ is a stochastic matrix (as $\Phi$ is positive and trace-preserving) \cite[Ch.~8.7]{HJ1}.
    \begin{itemize}
        \item[Step 2]: Next, we show that (\ref{eq:thm_main_1}) holds for all $\Phi\in\mathcal L(\mathbb C^{d\times d})$ 2-positive and trace-preserving for which there exists $\omega>0$ such that $\Phi(\omega)=\omega$.
    \end{itemize}
Because $\omega>0$, $\langle A,B\rangle_\omega:={\rm tr}(A^\dagger \omega^{-1/2}B\omega^{-1/2})$ is an inner product on $\mathbb C^{d\times d}$.
Now consider the adjoint $\Phi^\dagger$ of $\Phi$ with respect to $\langle\cdot,\cdot\rangle_{\rm HS} (=\langle\cdot,\cdot\rangle_{I})$, as well as the adjoint $\Phi^\#$ of $\Phi$ with respect to $\langle\cdot,\cdot\rangle_{\omega}$. Using the basic relation \mbox{$\langle A,B\rangle_\omega=\langle A,\omega^{-1/2}B\omega^{-1/2}\rangle_{\rm HS}$} and the cyclic property of the trace, one readily verifies
\begin{equation}\label{eq:Phi_sharp}
    \Phi^\#=\omega^{1/2} \Phi^\dagger(\omega^{-1/2}(\cdot)\omega^{-1/2}) \omega^{1/2}\,.
\end{equation}
Observe that $\Phi^\#$ is 2-positive (because $\Phi$ is), trace-preserving (because $\Phi(\omega)=\omega$), and satisfies $\Phi^\#(\omega)=\omega$ (because $\Phi$ is trace-preserving, which is equivalent to $\Phi^\dagger(I) = I$)---hence the same is true for $\tilde\Phi:=\frac12(\Phi+\Phi^\#)$. But now $\tilde\Phi$ is a self-adjoint element of the Hilbert space $(\mathcal L(\mathbb C^{d\times d},\langle\cdot,\cdot\rangle_\omega)$ meaning that all eigenvalues of $\tilde\Phi$ are real. Hence, by Step~1 we know that ${\rm tr}(\tilde\Phi)\leq d  \min\Re(\sigma(\tilde\Phi))+d^2-d$.
To pass over from $\tilde\Phi$ to $\Phi$ we make two observations:
\begin{align*}
    {\rm tr}(\Phi^\#)&={\rm tr}\big(\big(\omega^{1/2} (\cdot) \omega^{1/2}\big)\circ \Phi^\dagger\circ \big(  \omega^{-1/2}(\cdot)\omega^{-1/2}\big)\big)\\
    &={\rm tr}\big(\Phi^\dagger\circ \big(  \omega^{-1/2}(\cdot)\omega^{-1/2}\big)\circ\big(\omega^{1/2} (\cdot) \omega^{1/2}\big)\big)={\rm tr}(\Phi^\dagger)=\overline{{\rm tr}(\Phi)}={\rm tr}(\Phi)
\end{align*}
(in the first step the cyclic property of the trace is invoked, and in the last step we used that ${\rm tr}(\Phi)\in\mathbb R$ because $\Phi$ is Hermitian-preserving), as well as the Bendixson-Hirsch inequality \cite{Bendixson02,Hirsch02}, \cite[Thm.~12.6.6]{Mirsky55} \footnote{
Beware that this inequality holds for arbitrary finite-dimensional complex Hilbert spaces: if $v$ is a normalized eigenvector of $A$ to some $\lambda$, then $$
\Re(\lambda)=\Big\langle v\Big|\frac{A+A^\dagger}{2}\Big|v\Big\rangle\geq \min_{\langle x,x\rangle=1}\Big\langle x\Big|\frac{A+A^\dagger}{2}\Big|x\Big\rangle=\min\sigma\Big(\frac{A+A^\dagger}{2}\Big)
$$
(the last step follows, e.g., via unitary diagonalization of self-adjoint operators).
In particular, this applies to $(\mathcal L(\mathbb C^{d\times d}),\langle\cdot,\cdot\rangle_\omega)$.
}
$$
\min\Re\Big(\sigma\Big(\frac12(\Phi+\Phi^\#)\Big)\Big)\leq \min\Re(\sigma(\Phi))\,.
$$
Putting things together,
\begin{align*}
    {\rm tr}(\Phi)=\frac12\big( {\rm tr}(\Phi)+{\rm tr}(\Phi^\#) \big)={\rm tr}(\tilde\Phi)
    &\leq d  \min\Re(\sigma(\tilde\Phi))+d^2-d\\
    &\leq d  \min\Re(\sigma(\Phi))+d^2-d\,.
\end{align*}
    \begin{itemize}
        \item[Step 3]: Finally, we show that (\ref{eq:thm_main_1}) holds for all $\Phi\in\mathcal L(\mathbb C^{d\times d})$ 2-positive and trace-preserving
    \end{itemize}
    The key here is that set of all 2-positive and trace-preserving maps which have a full-rank fixed point is dense in the set of all 2-positive and trace-preserving maps $\Phi$:
    Indeed, $\Phi_\lambda:=(1-\lambda)\Phi+\lambda{\rm tr}(\cdot)\frac{I}{d}$ approximates $\Phi$ (as $\Phi=\lim_{\lambda\to 0}\Phi_\lambda$), and $\Phi_\lambda$ has a full-rank fixed point for all $\lambda\in(0,1]$ because for all non-zero $X\geq 0$ one has $\Phi_\lambda(X)=(1-\lambda)\Phi(X)+\lambda \frac{{\rm tr}(X)}{d}{I}>0$, so the same has to be true for any fixed point of $\Phi_\lambda$ (which necessarily exists, cf.~\cite[Thm.~4.24]{Watrous18}).
    This shows that there exists a sequence $(\Phi_n)_{n\in\mathbb N}$ of 2-positive, trace-preserving maps which all have a full-rank fixed point such that $\Phi=\lim_{n\to\infty}\Phi_n$.
    By Step~2, ${\rm tr}(\Phi_n)\leq d  \min\Re(\sigma(\Phi_n))+d^2-d$ so taking the limit and using that the eigenvalues depend continuously on the input \cite[Ch.~II, Thm.~5.14]{Kato80} concludes the proof.
\end{proof}

\subsection{Proof of Theorem~\ref{thm_main2}}

Transferring Theorem~\ref{thm_main} to the corresponding generators now is straightforward:
\begin{proof}[Proof of Theorem~\ref{thm_main2}]
    Let $L\in\mathcal L(\mathbb C^{d\times d})$ be given such that $e^{tL}$ is 2-positive and trace-preserving for all $t\geq 0$. In particular, we can express $L$ as the limit $L=\lim_{n\to\infty}\frac{e^{L/n}-{\rm id}}{1/n}$ and use Thm.~\ref{thm_main} (i.e., $e^{L/n}$ is 2-positive and trace-preserving for all $n$) to compute
    \begin{align*}
        {\rm tr}(L)&=\lim_{n\to\infty}n({\rm tr}(e^{L/n})-d^2)\\
        &\leq \lim_{n\to\infty}n(d  \min\Re(\sigma(e^{L/n}))+d^2-d-d^2)\\
        &=d\lim_{n\to\infty}n(  \min\Re(\sigma(e^{L/n}))-1)\\
        &=d\lim_{n\to\infty}\min\Re\Big(\sigma\Big(\frac{e^{L/n}-{\rm id}}{1/n}\Big)\Big)=d\min\Re(\sigma(L))\,.
    \end{align*}
    Again, we used continuity of the trace as well as of the spectrum \cite[Ch.~II, Thm.~5.14]{Kato80}.
\end{proof}

In principle one could also translate the proof of Thm.~\ref{thm_main} to the generator case; however, some of the steps of the proof become more difficult in that case, which is why we opted for the easier route via the 2-positive trace-preserving maps here.
For a direct proof of the generator case without going via 2-positive maps refer to the companion paper \cite{CEKM25} which, notably, does not prove the 2-positive \textit{map} case (Thm.~\ref{thm_main}).

\section{Outlook}\label{sec_outlook}

We showed that the recently identified spectral constraint \eqref{eq:relax_rates} on the relaxation rates of quantum dynamics really is a manifestation of 2-positivity, rather than complete positivity. Moreover, as seen in Remark~\ref{rem_boundtight}~(iii) this bound is tight, even for completely positive maps. 
Together with the fact that the constraint is attained by some completely positive dynamics \cite{CKKS21}, this implies that the allowed regions of relaxation rates for all $k$-positive dynamics share the common supporting hyperplane defined by a linear inequality~\eqref{eq:relax_rates}.
This naturally raises the question: is there any way to \textit{spectrally} distinguish complete positivity from, say, 2-positivity? 
If such a distinction is possible, what form would the resulting constraints on the relaxation rates take? At the very least one might expect them to appear in a nonlinear form.

Another open question is to what degree trace-preservation is needed for the bound at the core of this work to be valid? As explained at the beginning of Sec.~\ref{sec:main}, the main influence of trace-preservation on the inequality itself was that the largest eigenvalue was always 0 (in the generator case), resp.~1 (in the 2-positive map case).
Hence one may wonder whether all (generators of) completely positive maps, resp.~2-positive maps satisfy
\begin{equation}\label{eq:conj}
    {\rm tr}(\Phi)\leq d\min\Re(\sigma(\Phi))+(d^2-d)\max\Re(\sigma(\Phi))\,?
\end{equation}
In fact, all these classes of maps are covered by a generalization of ``conditional complete positivity'' (which guarantees that the one-parameter semigroup generated by some $L$ is completely positive for all times) \cite{Evans77,Wolf08b}: A map $L$ is called \textit{conditionally 2-positive} if $e^{tL}$ is 2-positive for all $t\geq 0$. 
In particular, every completely positive, every conditionally completely positive, and every 2-positive map is conditionally 2-positive.
With this,~\eqref{eq:conj} becomes really a conjecture about the spectrum of arbitrary conditionally 2-positive maps.
Preliminary numerics for the completely positive maps seem to support validity of~\eqref{eq:conj}; however, even if this is true, it would likely necessitate a vastly different approach because our proof fundamentally relied on the existence of a fixed point (which can no longer be guaranteed for maps that are just completely positive). 

Finally, it would be interesting to find examples of positive trace-preserving maps which violate~\eqref{eq:thm_main_1} as soon as $d\geq 3$ (recall also Remark ~\ref{rem_boundtight}~(ii)). Because numerics suggest that any violation of this bound in this case has to come about via a \textit{non-decomposable} positive map, further pursuing this question may lead to the discovery of new maps of this type, or perhaps even a way to numerically sample such maps (which, to the best of our knowledge, remains an open problem).

\section*{Acknowledgments}
FvE is funded by the \textit{Deutsche Forschungsgemeinschaft} (DFG, German Research Foundation) -- project number 384846402, and supported by the Einstein Foundation (Einstein Research Unit on Quantum Devices) and the MATH+ Cluster of Excellence. DC was supported by the Polish National Science Center
under Project No. 2018/30/A/ST2/00837.
DC was supported by the Polish National Science Center
under Project No. 2018/30/A/ST2/00837.
G.K. was supported by JSPS KAKENHI Grants Nos. 24K06873.

\appendix
\section{Examples and Counterexamples}\label{app_a}

\begin{example}\label{ex_ttc}
Given $d\geq 2$, $\alpha\in[0,1]$, consider the trace-preserving map
\begin{align*}
    \Phi_\alpha:\mathbb C^{d\times d}&\to\mathbb C^{d\times d}\\
    X&\mapsto\frac{\alpha(d-1){\rm tr}(X){I}-\alpha X+(1-\alpha)X^T}{1+\alpha(d-2)(d+1)}\,.
\end{align*}
This has first been considered by Takasaki and Tomiyama \cite{TT83}, based on an earlier work of Choi \cite{Choi72}. Indeed, Takasaki and Tomiyama have shown that $\Phi_\alpha$ is positive but not 2-positive if $\alpha\in[0,\frac1d)$, completely positive if $\alpha\in[\frac1d,\frac12]$, and $(d-1)$-positive but not completely positive if $\alpha\in(\frac12,1]$.
One readily computes
\begin{align*}
    {\rm tr}(\Phi_\alpha)&=\frac{\alpha(d-1){\rm tr}\big({\rm tr}(\cdot){I}\big)-\alpha {\rm tr}({\rm id})+(1-\alpha){\rm tr}({(\cdot)}^T)}{1+\alpha(d-2)(d+1)}\\
    &=\frac{\alpha (d-1)d-\alpha d^2+(1-\alpha)d}{1+\alpha(d-2)(d+1)} =\frac{(1-2\alpha)d}{1+\alpha(d-2)(d+1)}.
\end{align*}
Moreover, $\Phi$
has simple eigenvalue 
$1$, $\frac{d(d-1)}2$-fold eigenvalue
$-\frac{1}{1+\alpha(d-2)(d+1)}$,
and $\frac{(d-1)(d+2)}2$-fold eigenvalue
$\frac{1-2\alpha}{1+\alpha(d-2)(d+1)}$.
Inserting everything into~(\ref{eq:thm_main_1}) and re-arranging things yields the inequality
$$
(1-2\alpha)d\leq d\cdot(-1)+(d^2-d)\big( 1+(d-2)(d+1)\alpha \big)
$$
which after a straightforward computation is equivalent to
\begin{equation}\label{eq:alpha_violate}
    \alpha\geq\frac{3-d}{(d^2-1)(d-2)+2}\,.
\end{equation}
For $d\geq 3$ this is trivially fulfilled because $\alpha\geq 0$. For $d=2$, \eqref{eq:alpha_violate} reduces to $\alpha\geq\frac12$ meaning the desired inequality is violated for $\alpha\in[0,\frac12)$ (i.e., for $d=2$, $\Phi_\alpha$ serves as a counterexample to Thm.~\ref{thm_main} when relaxing 2-positivity to positivity, because in the positive regime $\alpha\in[0,\frac12)$ the inequality~(\ref{eq:thm_main_1}) does not hold).
Combining these facts shows that, in particular, the transpose map ($\alpha=0$) is a counterexample to the relaxed version of Thm.~\ref{thm_main} (if and) only if $d=2$.
\end{example}

\begin{example}\label{ex_bound_tight}
    Let $d\in\mathbb N$, $c>d$. We start by observing that the function
\begin{align*}
    f:[0,1]&\to\mathbb R\\
    x&\mapsto (4-2c)x^2+(4d-8)x+(4-4d+2c)
\end{align*}
satisfies $f(1)=0$ and $f'(1)=4(d-c)<0$ meaning there exists $x\in(0,1)$ such that $f(x)>0$.
To construct $\Phi$ let $\alpha\in(0,\frac\pi2)$ be such that $\cos(\alpha)=x$.
We define $H:=\alpha(|0\rangle\langle 0|-|1\rangle\langle 1|)$, $U:=e^{iH}$, and $\Phi:=U(\cdot)U^\dagger $ ($\Phi$ is 2-positive and trace-preserving because $U$ is unitary).
One readily verifies that 
$\Phi$ has eigenvalues
$\{1,e^{2i\alpha},e^{2i\alpha}\}$ if $d=2$ and $\{1,e^{i\alpha},e^{2i\alpha},e^{-i\alpha},e^{2i\alpha}\}$ if $d\geq 3$; in either case, because $\alpha\in(0,\frac\pi2)$ one has
$$
\min \Re(\sigma(\Phi))=\cos(2\alpha)=2\cos^2(\alpha)-1=2x^2-1\,.
$$
Also 
${\rm tr}(\Phi)=|{\rm tr}(U)|^2=(d-2+2\cos(\alpha))^2=(d-2+2x)^2$. Putting things together, (\ref{eq:optimal_bound}) 
evaluates to $(d-2+2x)^2>c(2x^2-1)+d^2-c$. This is easily seen to
be equivalent to $f(x)>0$---which holds by construction---meaning that $\Phi$ satisfies~(\ref{eq:optimal_bound}), as desired.
\end{example}

\begin{example}\label{ex_2p}
    Consider the qubit map $\Phi(X):=4X_{11}e_1e_1^T+\sigma_xX^T\sigma_x$ so the Choi matrix of $\Phi$ reads
    $$
        \mathsf C(\Phi)=\begin{pmatrix}
            4&0&0&1\\0&1&0&0\\0&0&1&0\\
        1&0&0&0
        \end{pmatrix} \,.
    $$
This $\Phi$ is positive but not 2-positive, and
${\rm tr}(\Phi)=
6$ as is readily verified.
On the other hand, let $U\in\mathbb C^{2\times 2}$ be an arbitrary special unitary matrix, i.e., there exist $\phi,\omega,\xi\in\mathbb R$ such that
\begin{equation*}
    U=\begin{pmatrix}
        \cos(\phi)e^{i\xi}& \sin(\phi)e^{i\omega} \\
-\sin(\phi)e^{-i\omega}  & \cos(\phi)e^{-i\xi}
    \end{pmatrix} .
\end{equation*}
A straightforward computation shows that the corresponding transition matrix reads
$$
T_U=\begin{pmatrix}
    4\cos(\phi)^2&1\\
1&4\sin(\phi)^2
\end{pmatrix} \,.
$$
Hence ${\rm tr}(\Phi)=6\leq 8=2\cdot 4=d\,{\rm tr}(T_G(\Phi))$ for all orthonormal bases $G$ of $\mathbb C^2$, so $\Phi$ satisfies~(\ref{eq:trPhiTG}) for all orthonormal bases without being 2-positive.
\end{example}

  \bibliographystyle{elsarticle-num} 
  \bibliography{control21vJan20}



%
%
%
\end{document}